\documentclass[11pt,a4paper]{article}
\usepackage{amsmath}
\usepackage{amssymb}
\usepackage{amsfonts}
\usepackage{amsthm}
\usepackage{geometry}
\usepackage{mathrsfs}
\usepackage{graphicx}
\usepackage{tikz}
\usetikzlibrary{decorations.pathreplacing,arrows,calc,intersections,through,backgrounds}
\usepackage{float}
\usepackage{booktabs}
\usepackage[colorlinks=true,linkcolor=blue,citecolor=red]{hyperref}

\newtheorem{theorem}{Theorem}[section]
\newtheorem{lemma}[theorem]{Lemma}

\newtheorem{definition}[theorem]{Definition}

\newtheorem{proposition}{Proposition}[section]

\theoremstyle{definition}

\begin{document}

\title{Limit Values of Character Sums in Frobenius Formula of Three Permutations}
\author{Dun Liang \and Bin Xu\thanks{ B.X. is supported in part by the Project of Stable Support for Youth Team in Basic Research Field, CAS (Grant No. YSBR-001) and NSFC (Grant No. 12271495).} \and Wenyan Yang$^\dagger$}
\date{}
\maketitle

\begin{abstract}
We study the asymptotic behaviour of the character summation part of the  Frobenius formula for three conjugacy classes of the symmetric groups.  
When two conjugacy classes contain fixed points of order $\sim H_i\sqrt{n}$ for $i=1,2$ and all other cycles are long, the summation converges uniformly to $2e^{-H_1H_2}$; the exponential factor coincides with the non-collision probability of the birthday paradox.  
When short cycles are rare in all three classes, character sum tends to $2$.   
\end{abstract}

\section{Introduction}

The classical Hurwitz existence problem asks whether there  exists branched
coverings of the Riemann sphere with prescribed branching data
\cite{Hurwitz1891,Edmonds1984,Pervova2006}.
The case of three branch points is especially important: after normalising
them to $\{0,1,\infty\}$, it corresponds to Belyi maps and dessins d'enfants
\cite{grothendieck1997,LanZvo}.
Let $C_1,C_2,C_3$ be three conjugacy classes of the symmetric group
$\mathfrak{S}_n$ and let $\mathscr{Y}_n$ be the set of Young diagrams
(or partitions) of $n$.
The Frobenius formula (see e.g.\ \cite[page10, Theorem 1.1.12]{LanZvo})
\begin{equation}\label{Frob}
\#\{\, (\sigma_1,\sigma_2,\sigma_3)\in C_1\times C_2 \times C_3 \mid
\sigma_1\sigma_2\sigma_3 = \mathrm{id} \}
= \frac{|C_1|\cdot |C_2|\cdot |C_3|}{n!}\,\mathbf{Y}_n(C_1,C_2,C_3)
\end{equation}
where
\begin{equation}\label{Yn}
\mathbf{Y}_n(C_1,C_2,C_3)
= \sum_{\lambda \in \mathscr{Y}_n}
\frac{\chi^\lambda(C_1)\chi^\lambda(C_2)\chi^\lambda(C_3)}
{\chi^\lambda(\mathbf{1})}
\end{equation}
reformulates the first necessary condition of the Hurwitz problem as a
question in representation theory.
In this paper we study the asymptotic behaviour of the character sum
$\mathbf{Y}_n$ when two of the conjugacy classes contain a large number of
fixed points. We prove that if $C_1$ and $C_2$ have approximately
$H_1\sqrt{n}$ and $H_2\sqrt{n}$ fixed points respectively and all other
cycles are long, then $\mathbf{Y}_n$ converges to $2e^{-H_1H_2}$.

The representation-theoretic approach to the Hurwitz problem goes back
to the Frobenius formula~\eqref{Yn}.  Mednykh~\cite{mednykh1981}
derived an exact formula purely in terms of character sums; the
resulting expression, however, is complicated and unsuitable for
asymptotic analysis.  On the computational side, Zheng~\cite{Zheng2006}
and Xu~\cite{Wang2025} have used direct calculation to produce exceptional data for
small~$n$.  None of these works addresses the asymptotic behaviour of character sums.  In
contrast, the asymptotic behaviour of symmetric group characters
itself has seen major progress through the work of
Larsen--Shalev~\cite{LarsenShalev2008}, Roichman~\cite{Roichman},
Liebeck--Shalev~\cite{Liebeck2004}, and others. We will use these results later in this work.

Let $H_1,H_2,P$ be three positive real numbers. Let $m_{1,n}$ and $m_{2,n}$  be two positive integer sequences  which converge to $+\infty$ of order $H_1\sqrt{n}$ and $H_2\sqrt{n}$, respectively. Our main theorem, Theorem \ref{SimiGaussian}, shows that the expression (\ref{Yn}) converges to the limit value $2{\rm e}^{-H_1H_2}$ for any triple $C_1,C_2,C_3$ of conjugacy classes of ${\frak S}_n$, whenever $C_1$ has $m_{1,n}$ fixed points, and $C_2$ has $m_{2,n}$ fixed points and all other cycles of all three conjugacy classes have lengths greater or equal to $P \sqrt{n} \log n$. This convergence is uniform, independent from the specific shapes of $C_1,C_2,C_3$.

 The proof  relies on a sequence of subtle
analytic and combinatorial estimates.  Nevertheless, the limit
$e^{-H_1H_2}$ has an elementary probabilistic interpretation:
in the birthday paradox \cite{Mitzenmacher}, the probability that
two random subsets of sizes $H_1\sqrt{n}$ and $H_2\sqrt{n}$ in an
$n$-element set fail to intersect converges precisely to
$e^{-H_1H_2}$.

\subsection{Main Results and Hurwitz Problem}

For each \(i \geq 1\), define the $i$-{\bf th cycle length function} to be the class function \(\theta_i: \mathfrak{S}_n \to \mathbb{Z}_{\geq 0}\) by
\begin{equation}\label{thetai}
\theta_i(\sigma) = \text{number of } i\text{-cycles in } \sigma.
\end{equation}
If $C$ is a conjugacy class of ${\mathfrak S}_n$, then $\theta_i(C)=\theta_i(\sigma)$ for any $\sigma \in C$. We denote the conjugacy class of the identity element by ${\bf 1}$.
A permutation (or conjugacy class) of ${\frak S}_n$ is called $r$-{\bf derangement} (see Example II.14 in \cite{FS}) if $\theta_i(\sigma)=0$ (or $\theta_i(C)=0$) for all $i\leq r$.

Let $\varepsilon(C)$ be the sign of $C$ whose value is 1 or -1 depends on whether $C$ is even or odd.  In order to make the set \(
\{\, (\sigma_1,\sigma_2,\sigma_3)\in C_1\times C_2 \times C_3 \ | \ \sigma_1\sigma_2\sigma_3 = {\rm id} \}
\) non-empty, it is necessary that $\varepsilon(C_1)\varepsilon(C_2)\varepsilon(C_3)=1$. 

\begin{theorem}[The Semi-Gaussian Law]\label{SimiGaussian}Fix real numbers $H_1>0$, $H_2>0$ and $P>0$. Let $\{m_{1,n}\}_{n \in {\mathbb N}}$ and $\{m_{2,n}\}_{n \in {\mathbb N}}$ be two sequences of non-negative integers such that 
\begin{equation}
\lim_{n\rightarrow \infty}\frac{m_{1,n}}{\sqrt{n}}=H_1, \qquad \lim_{n\rightarrow \infty}\frac{m_{2,n}}{\sqrt{n}}=H_2.
\end{equation}
Let $Q_n=P\sqrt{n} \log n$. Then for any $\varepsilon >0$, there exists $N>0$, such that for any $n>N$, for any three conjugacy classes $C_{1,n},C_{2,n},C_{3,n}$ of ${\mathfrak{S}}_n$, such that
\begin{enumerate}
\item $\theta_1(C_{1,n})=m_{1,n}, \quad \theta_1(C_{2,n})=m_{2,n}$;
\item For each $2\leq j <Q_n$, $\theta_j(C_{1,n})=
\theta_j(C_{2,n})=0$ ;
\item For each $1 \leq j < Q_n$, $\theta_j(C_{3,n})=0,$; 
\item $\varepsilon(C_{1,n})\varepsilon(C_{2,n})\varepsilon(C_{3,n})=1$,
\end{enumerate}
we have
$$|{\bf Y}_n(C_{1,n},C_{2,n},C_{3,n})-2{\rm e}^{-H_1H_2}|<\varepsilon.$$
\end{theorem}

In Theorem \ref{SimiGaussian}, we use the $\varepsilon$-$\delta$ language to state the uniformity of the convergence of the conjugacy classes. We emphasize that  the $C_{i,n}$ could be absent for some $n$, the statement implies that no matter what the shapes of $C_{i,n}$ are, the approximation holds once the expression of ${\bf Y}_n$ exists. This theorem reveals the phenomenon that the character sum could attain any positive real number in the interval $(0,2)$ when the two among the three conjugacy classes attain the level of $\sqrt{n}$. Contrarily, when fixed points are rare, the character sum will converge to the number 2 despite of the shape of the Young diagrams. In fact we can show the following proposition.

\begin{proposition}\label{limit2} For each $n$, let $(C_{1,n},C_{2,n},C_{3,n})$ be a triple of conjugacy classes of ${\frak S}_n$. If 
for all $i=1,2,3$, the conjugacy class $C_{i,n}$ has most $n^{o(1)}$ cycles of length less
than 4, and $\varepsilon(C_{1,n})\varepsilon(C_{2,n})\varepsilon(C_{3,n})=1$, then
$$\lim_{n\rightarrow \infty}{\bf Y}(C_{1,n},C_{2,n},C_{3,n})=2.$$
\end{proposition}

 From this perspective, both Theorem \ref{SimiGaussian} and Proposition \ref{limit2} provide affirmative answers to the first condition of the Riemann existence theorem.

\subsection{ Heuristic explanation of Theorem \ref{SimiGaussian} and Proposition \ref{limit2}}
The limits appearing in Theorem \ref{SimiGaussian} and Proposition \ref{limit2} have   a natural probabilistic interpretation. Although the following argument is heuristic rather than a proof, it provides an intuitive explanation for the constants appearing in the corresponding results.

Consider the uniform distribution on the set $C_1\times C_2\times C_3$. The probability that a random triple $(\sigma_1,\sigma_2,\sigma_3)$ satisfies
\(
\sigma_1\sigma_2\sigma_3=\mathrm{id}
\)
is
\[
\frac{\#\{(\sigma_1,\sigma_2,\sigma_3)\in C_1\times C_2\times C_3
\mid \sigma_1\sigma_2\sigma_3=\mathrm{id}\}}
{|C_1||C_2||C_3|}.
\]

Since the condition $\sigma_1\sigma_2\sigma_3=\mathrm{id}$ is equivalent to
$
\sigma_3=(\sigma_1\sigma_2)^{-1},
$
we may first choose $\sigma_1\in C_1$ and $\sigma_2\in C_2$ at random and then determine $\sigma_3$. Thus the problem reduces to estimating the probability that $(\sigma_1\sigma_2)^{-1}$ lies in $C_3$.

For Theorem \ref{SimiGaussian}, instead of considering the limit properties of the conjugacy classes, we suppose that
$
\theta_1(C_1)=H_1\sqrt{n}$ and $\theta_1(C_2)=H_2\sqrt n,
$ which are integers. This assumption naturally implies that $\lim_{n\rightarrow \infty}m_{i,n}/\sqrt{n}=H_i$ for $i=1,2$ where $\{m_{i,n}\}$ as in Theorem \ref{SimiGaussian}. 
Let
$
A=Fix(\sigma_1),  B=Fix(\sigma_2)
$
be the set of fixed points of $\sigma_1$ and $\sigma_2$.
Then
$
|A|=H_1\sqrt{n}$ and $|B|=H_2\sqrt n.
$ If a point $i$ belongs to both $A$ and $B$, then
$
\sigma_1(i)=i,  \sigma_2(i)=i,
$
which implies
$
\sigma_1\sigma_2(i)=i.
$
Thus $i$ becomes a fixed point of $\sigma_3=(\sigma_1\sigma_2)^{-1}$. However, under the conditions of Theorem \ref{SimiGaussian} the class $C_3$ is required to be a large derangement, so such fixed points must be avoided. Therefore we need
$
A\cap B=\varnothing.
$

The problem is therefore heuristically reduced to estimating the probability that two random subsets of size $H_1\sqrt n$  and $H_2\sqrt n$ in an $n$-element set have empty intersection. A direct computation shows
\[
P(A\cap B=\varnothing)
=
\frac{n-H_1\sqrt{n}}{n}\cdot \frac{n-H_1\sqrt{n}-1}{n-1} \cdot \cdots \cdot \frac{n-H_1\sqrt{n}-H_2\sqrt{n}+1}{n-H_2\sqrt{n}+1}.
\]

As $n\to\infty$, this probability satisfies
\[
P(A\cap B=\varnothing)\rightarrow {\rm e}^{-H_1H_2}.
\]
This asymptotic behaviour is analogous to the classical {\bf birthday paradox} (see \cite[page 90, Section 5.1]{Mitzenmacher}): when about $\sqrt n$ objects are chosen from a set of size $n$, the probability of collisions approaches a nontrivial constant, while the probability of having no collision converges to a nontrivial exponential constant.

Therefore
$
P(\sigma_3\in C_3)\rightarrow  e^{-H_1H_2},
$ as $n\rightarrow +\infty$
. By~\eqref{Frob},
\[
{\bf Y}_n
=
\frac{n!}{|C_1||C_2||C_3|}
\#\{(\sigma_1,\sigma_2,\sigma_3)\in C_1\times C_2\times C_3
\mid \sigma_1\sigma_2\sigma_3=\mathrm{id}\}.
\]
Since $\sigma_3=(\sigma_1\sigma_2)^{-1}$ is forced by the sign condition
to have prescribed parity, it ranges over only $n!/2$ permutations.
Assuming that $(\sigma_1\sigma_2)^{-1}$ is roughly uniformly distributed 
among the $n!/2$ permutations with the prescribed parity, and combining with the
birthday-paradox probability $e^{-H_1H_2}$, we obtain
\[
{\bf Y}_n
\;\approx\;
\frac{n!}{|C_1||C_2||C_3|}
\cdot |C_1||C_2|\cdot e^{-H_1H_2}\cdot\frac{|C_3|}{n!/2}
=
2e^{-H_1H_2}.
\]

For Proposition \ref{limit2}, when $n$ is large, it is natural to assume heuristically that the product of two permutations behaves approximately like a random permutation, except that its parity is already determined. Hence $(\sigma_1\sigma_2)^{-1}$ may be regarded as being roughly uniformly distributed among the permutations with the prescribed parity. Since there are about $n!/2$ permutations with a fixed parity, the probability that $\sigma_3\in C_3$ is approximately
\(
\frac{|C_3|}{n!/2}.
\)
Consequently, the number of triples satisfying
$
\sigma_1\sigma_2\sigma_3=\mathrm{id}
$
is roughly
\[
|C_1||C_2|\frac{|C_3|}{n!/2}.
\]
After the normalization used in the definition of ${\bf Y}_n(C_1,C_2,C_3)$, this leads to
\[
{\bf Y}_n(C_1,C_2,C_3)\rightarrow
n!\cdot\frac{2}{n!}=2,
\]
which explains the limiting value in Proposition \ref{limit2}.

\subsection{Sketch of the Proof}

The starting point for both Theorem \ref{SimiGaussian} and Proposition \ref{limit2} is a decomposition of the Frobenius character sum.  We partition the set of Young diagrams by the defects of the first row and the first column.  For any real cut-off $k\le (n-1)/2$, Lemma~\ref{ksep} gives the disjoint decomposition
\(
 {\bf Y}_n=2\sum_{0\le t<k}{\bf Z}_n^t+{\bf X}_n^k,
\)
which follows from the sign condition $\varepsilon(C_1)\varepsilon(C_2)\varepsilon(C_3)=1$ and the fact that conjugation of Young diagrams preserves the absolute value of the summand under this sign condition.

For Theorem~\ref{SimiGaussian} we take $k=Q_n=P\sqrt n\log n$.
The main contribution comes from the layers $0\le t<Q_n$.
By Lemmas~\ref{important} and~\ref{chi1t}, the absence of short
cycles in $C_3$ implies that among the diagrams in
${\mathscr Y}_n^t$ only the hook $(n-t,1^t)$ contributes, with value
$(-1)^t$.  For $C_1$ and $C_2$, the Newton's identities for exterior powers Lemma~\ref{wedgeformula}
 yields the explicit hook
values.  Consequently the main layers sum to a term $q_n(t)$.
By Lemma~\ref{SG-An-uniform}, for $0\le t\le n^{2/9}$ one may write the character sum as $\sum_{0\le t<Q_n}q_n(t)$ where
$q_n(t)=(-1)^t\mu_n^t A_n(t)/t!$, in which $\mu_n$ captures the
macroscopic rate and $A_n(t)$ is a local correction factor.
The former converges to $H_1H_2$ and the latter to $1$ uniformly;
hence this part converges to the Taylor series
$\sum_{t\ge0}(-1)^t(H_1H_2)^t/t!=e^{-H_1H_2}$
(Lemma~\ref{SG-taylor-main}).  The remaining range
$n^{2/9}<t<Q_n$ is negligible because the factorial in the
denominator suppresses all terms (Lemma~\ref{SG-middle-tail}).

The tail ${\bf X}_n^{Q_n}$ is handled by a sharp dimension lower bound.  Lemma~\ref{two-row-min-degree} gives the minimal dimension among diagrams in ${\mathscr Z}_n^t$, and Lemma~\ref{uniform-exponential-lower-bound} provides a uniform exponential lower bound for all diagrams whose first row and first column are bounded by a fixed fraction of $n$.  Together they yield $\chi^{\lambda}({\bf 1})\ge D_n$ with $\log D_n\ge \frac P4\sqrt n(\log n)^2$ for every $\lambda\in{\mathscr X}_n^{Q_n}$.  On the other hand, the Larsen--Shalev orbit-growth invariant for a permutation in $C_2$ satisfies $E(\sigma)\le 1/2+o(1)$ because all non-fixed cycles have length at least $Q_n$.  By Larsen--Shalev~\cite[Theorem~1.1(i)]{LarsenShalev2008} this gives the uniform ratio estimate
\(
 \bigl|\chi^{\lambda}(C_2)/\chi^{\lambda}({\bf 1})\bigr|\le (\chi^{\lambda}({\bf 1}))^{-3/8}.
\)
The other two character factors are treated globally by the Cauchy--Schwarz inequality and the column orthogonality relations.  The resulting bound for $|{\bf X}_n^{Q_n}|$ decays like $\exp\{-c\sqrt n(\log n)^2\}$, which dominates the centralizer factors and hence tends to zero.  Combining the main layers and the tail gives ${\bf Y}_n\to 2e^{-H_1H_2}$.

For Proposition~\ref{limit2} we use a different strategy.  Since each conjugacy class has only $n^{o(1)}$ cycles of length less than $4$, the Larsen--Shalev estimate~\cite[Theorem~1.2(ii)]{LarsenShalev2008} with $m=4$ gives $|\chi^{\lambda}(C_i)|\le (\chi^{\lambda}({\bf 1}))^{1/4+\delta}$ for all non-linear irreducible characters.  Multiplying the three estimates and dividing by $\chi^{\lambda}({\bf 1})$ bounds the summand by a negative power of the dimension.  The Witten zeta function $\sum_{\lambda}(\chi^{\lambda}({\bf 1}))^{-1/5}$ converges to $2$ by a theorem of Liebeck and Shalev~\cite{Liebeck2004}, and subtracting the trivial and sign contributions yields the limit $2$.

\subsection{Organization of the Paper}

This paper is organized as follows.  In Section 2 we introduce the layer notation for Young diagrams and prove the basic decomposition Lemma~\ref{ksep}.  We also collect several character computations that will be used throughout the paper: an alternative version of the Murnaghan--Nakayama rule (Theorem~\ref{AMN}), exact values for hook diagrams (Lemmas~\ref{important} and \ref{chi1t}), and Newton's identities for exterior powers (Lemma~\ref{wedgeformula}).

Section~\ref{section2} is devoted to the main layer limit.  In Section~\ref{section4o1} we compute the explicit contribution of the layers $0\le t<Q_n$ and show that it converges to ${\rm e}^{-H_1H_2}$ (Proposition~\ref{SG-main-layers-estimate}).  The argument combines the exact character formulas with a Taylor-expansion approximation and a combinatorial estimate for the intermediate range $n^{2/9}<t<Q_n$.

Section~\ref{section4o2} treats the tail approximation.  We first establish lower bounds for the dimensions of irreducible representations in the range ${\mathscr X}_n^{Q_n}$ (Lemmas~\ref{two-row-min-degree} and \ref{uniform-exponential-lower-bound}), then apply the Larsen--Shalev bound to control the character ratios, and finally use orthogonality relations to prove that the tail ${\bf X}_n^{Q_n}$ tends to zero (Proposition~\ref{SG-X-tail-estimate}).  Theorem~\ref{SimiGaussian} follows immediately from these two propositions.

Section~\ref{section3} proves Proposition~\ref{limit2}.  The argument relies on the Larsen--Shalev character bound for classes with few short cycles and the convergence of the Witten zeta function.  In the final section we discuss the implications of our results for the Riemann existence theorem and state a conjecture on the realizability of branching data by $3$-derangements.

\section{Regrouping of the Character Sums}

From now on we start to prove Theorem \ref{SimiGaussian}. 
We first introduce the layer notation for Young diagrams.  For any
$\lambda=(\lambda_1,\ldots,\lambda_l)\in {\mathscr Y}_n$, let
$\lambda'$ denote the conjugate Young diagram.

\begin{definition}\label{YZXsets}
For $0\leq t\leq n-1$, define
\[
 {\mathscr Y}_n^t
 =\{\lambda\in{\mathscr Y}_n\mid n-\lambda_1=t\},
 \qquad
 {\mathscr Z}_n^t
 =\{\lambda\in{\mathscr Y}_n\mid \lambda'\in{\mathscr Y}_n^t\}.
\]
Thus ${\mathscr Y}_n^t$ consists of diagrams whose first row has length
$n-t$, while ${\mathscr Z}_n^t$ consists of diagrams whose first column has
length $n-t$.  For a real number $k$, define
\[
 {\mathscr X}_n^k
 =\{\lambda\in{\mathscr Y}_n\mid n-\lambda_1\geq k
 \text{ and } n-\lambda'_1\geq k\}.
\]
Equivalently, ${\mathscr X}_n^k$ consists of diagrams whose first row and
first column both have length at most $n-k$, i.e., diagrams contained in the top-left $(n-k)\times (n-k)$-square.
\end{definition}

In Theorem \ref{SimiGaussian}, we emphasize that the $C_{i,n}$'s are conjugacy classes of ${\frak S}_n$ for a specific $n$. For simplicity, later in this paper, if $n$ is clear for the context, we will skip the $n$ in $C_{i,n}$  and simply denote $C_i$ for $i=1,2,3$.   Given three conjugacy classes $C_1,C_2,C_3$ of ${\mathfrak S}_n$, we use
bold letters for the corresponding character sums.  Thus
\[
 {\bf Y}_n(C_1,C_2,C_3)={\bf Y}_n
 =\sum_{\lambda\in{\mathscr Y}_n}
 \frac{\chi^\lambda(C_1)\chi^\lambda(C_2)\chi^\lambda(C_3)}
 {\chi^{\lambda}({\bf 1})},
\]
\[
 {\bf Z}_n^t(C_1,C_2,C_3)={\bf Z}_n^t
 =\sum_{\lambda\in{\mathscr Z}_n^t}
 \frac{\chi^\lambda(C_1)\chi^\lambda(C_2)\chi^\lambda(C_3)}
 {\chi^{\lambda}({\bf 1})},
\]
and
\[
 {\bf X}_n^k(C_1,C_2,C_3)={\bf X}_n^k
 =\sum_{\lambda\in{\mathscr X}_n^k}
 \frac{\chi^\lambda(C_1)\chi^\lambda(C_2)\chi^\lambda(C_3)}
 {\chi^{\lambda}({\bf 1})}.
\]
When the conjugacy classes are clear from the context, they will also be omitted from the notation.

\begin{lemma}\label{ksep}
Let $C_1,C_2,C_3$ be three conjugacy classes of ${\mathfrak S}_n$ such that
\(
 \varepsilon(C_1)\varepsilon(C_2)\varepsilon(C_3)=1.
\)
If $k$ is a real number with $k\leq(n-1)/2$, then
\[
 {\bf Y}_n
 =2\sum_{0\leq t<k}{\bf Z}_n^t+{\bf X}_n^k,
\]
where the sum is over all integers $t$ satisfying $0\leq t<k$.
\end{lemma}

\begin{proof}
For every partition $\lambda\in {\mathscr Y}_n$ and every conjugacy class $C$ of
${\mathfrak S}_n$, one has
\[
 \chi^{\lambda'}(C)=\varepsilon(C)\chi^\lambda(C),
 \qquad
 \chi^{\lambda'}({\bf 1})=\chi^{\lambda}({\bf 1}).
\]
Hence the summand associated with $\lambda'$ is equal to the summand associated with $\lambda$ whenever
\(\varepsilon(C_1)\varepsilon(C_2)\varepsilon(C_3)=1\).  Therefore, for every integer $t$ we have
\(
 {\bf Y}_n^t={\bf Z}_n^t,
\)
where ${\bf Y}_n^t$ denotes the character sum over ${\mathscr Y}_n^t$.

We now decompose the set of Young diagrams.  Let $\lambda\in{\mathscr Y}_n$.
If $n-\lambda_1<k$, then $\lambda$ lies in one of the layers
${\mathscr Y}_n^t$ with $0\leq t<k$.  If $n-\lambda'_1<k$, then $\lambda$ lies in one of the layers
${\mathscr Z}_n^t$ with $0\leq t<k$.  If neither inequality holds, then
$\lambda\in{\mathscr X}_n^k$.

The two unions of layers are disjoint.  Indeed, if for some integers
$s,t<k$ we had
\(
 \lambda\in{\mathscr Y}_n^s\cap{\mathscr Z}_n^t,
\)
then
\(
 \lambda_1>n-k,\) and \(
 \lambda'_1>n-k.
\)
Thus
\(
 \lambda_1+\lambda'_1>2n-2k\geq n+1,
\)
contradicting the elementary bound $\lambda_1+\lambda'_1\leq n+1$.
Therefore we have a disjoint union
\[
 {\mathscr Y}_n
 =\left(\bigsqcup_{0\leq t<k}{\mathscr Y}_n^t\right)
 \sqcup
 \left(\bigsqcup_{0\leq t<k}{\mathscr Z}_n^t\right)
 \sqcup
 {\mathscr X}_n^k.
\]
Taking character sums and using ${\bf Y}_n^t={\bf Z}_n^t$ gives the result.
\end{proof}

Take $k=Q_n=P\sqrt{n}\log n$ in Lemma \ref{ksep}, we have
\[
 {\bf Y}_n
 =2\sum_{0\leq t<Q_n}{\bf Z}_n^t+{\bf X}_n^{Q_n}.
\]

According to this equality, in order to show Theorem \ref{SimiGaussian}, we will show the following two propositions.

\begin{proposition}\label{layerestimate} 
Fix real numbers $H_1>0$, $H_2>0$ and $P>0$. Let $\{m_{1,n}\}_{n \in {\mathbb N}}$ and $\{m_{2,n}\}_{n \in {\mathbb N}}$ be two sequences of non-negative integers such that 
$
\lim_{n\rightarrow \infty}\frac{m_{1,n}}{\sqrt{n}}=H_1, ,\ \lim_{n\rightarrow \infty}\frac{m_{2,n}}{\sqrt{n}}=H_2.
$
Let $Q_n=P\sqrt{n} \log n$. Then for any $\eta > 0$, there exists $N=N(\eta, H_1,H_2, P)$, such that for all $n>N$ and for all triples satisfying conditions 1-4 in Theorem \ref{SimiGaussian}, we have
\[
 \left|\sum_{0\leq t<Q_n}{\bf Z}_n^t-e^{-H_1H_2}\right|<\eta.
\]
\end{proposition}

\begin{proposition}\label{SG-X-tail-estimate}
Under the conditions of Proposition \ref{layerestimate}, for any $\eta > 0$, there exists $N=N(\eta, H_1,H_2, P)$, such that for all $n>N$ and for all triples satisfying conditions 1-4 in Theorem \ref{SimiGaussian}, we have
\[
 |{\bf X}_n^{Q_n}|<\eta.
\]
\end{proposition}

\section{The Main Layer Limit}\label{section2}

In this section we prove Proposition \ref{layerestimate}. We need to compute some character values before we prove the main layer limit. In this paper we focus on explicit computations, hence we refer to \cite{GTM129,GTM203,MacDonald} as  elementary materials for representation theory of the symmetric groups. We use the original form of the Hardy--Ramanujan bound on the size of ${\mathscr Y}_n$.

\begin{theorem}[Hardy-Ramanujan \cite{HardyRamanujan1918}]\label{HRbound} There exists $K>0$ such that $$|{\mathscr Y}_n|<\frac{K}{n}\cdot {\rm e}^{2\sqrt{2}\sqrt{n}}.$$
\end{theorem}

\subsection{Computations on Character Values}

Throughout this paper, we will use several important lemmas on character values, computed using an alternative version of the Murnaghan--Nakayama rule.
Recall that a {\bf rim hook} (or border strip) is a connected skew Young diagram that contains no $2\times 2$ square. 
A {\bf rim hook tableau} is a generalized tableau \(T\) with positive integer entries such that the rows and columns of \(T\) weakly increase, and all occurrences of \(i\) in \(T\) lie in a single rim hook.
Now define the {\bf sign} of a rim hook tableau with rim hooks \(\xi^{(i)}\) to be
\[
(-1)^T = \prod_{\xi^{(i)} \in T} (-1)^{l(\xi^{(i)})}
\]
where $l(\xi^{(i)})$ is the {\bf leg length} of $\xi^{(i)}$, i.e., the number of columns that $\xi^{(i)}$ occupies minus 1.

\begin{theorem}[Alternative Version of the Murnaghan--Nakayama Rule, see Corollary 4.10.6 in \cite{GTM203}]
\label{AMN}
Let \(\lambda\) be a partition of \(n\) and let \(\alpha = (\alpha_1, \ldots, \alpha_k)\) be any composition of \(n\). Then
\[
\chi_\alpha^\lambda = \sum_T (-1)^T,
\]
where the sum is over all rim hook tableaux of shape \(\lambda\) and content \(\alpha\).
\end{theorem}
\begin{lemma}\label{important}
Let $C$ be a conjugacy class of ${\mathfrak S}_n$ such that all cycles of $C$ have length greater than or equal to $k$. Then for all $t\le k-1$ and all $\lambda\in{\mathscr Y}_n^t$, 
$$
\chi^\lambda(C)=0 \qquad \mbox{if }\quad \lambda\neq (n-t,1^t).
$$
\end{lemma}

{\em Proof.}
According to Theorem~\ref{AMN}, it suffices to show that no rim hook tableau $T$ of shape $\lambda$ and content $C$ can exist.
\begin{figure}[H]
\centering
% ============================================================
% 图 1：\lambda_2=2 的情形
% ============================================================
\begin{tikzpicture}[scale=0.55,every node/.style={font=\small}]

% --- 第一行 (最上方，长度 n-t) ---
\foreach \x in {0,1,2} {\draw[fill=gray!15] (\x,3) rectangle (\x+1,4);}
\draw[fill=gray!15] (3.5,3) rectangle (6.5,4);
\node at (5,3.5) {$\dots$};
\foreach \x in {7,8} {\draw[fill=gray!15] (\x,3) rectangle (\x+1,4);}
\node[above] at (4.5,4) {\footnotesize $\lambda_1 = n-t$};

% --- 第二行 (中间，长度 2) ---
\draw[fill=blue!15] (0,2) rectangle (1,3);
\draw[fill=red!30]  (1,2) rectangle (2,3);

% --- 第 3 行及以下 (首列，共 t-2 个) ---
\draw[fill=blue!15] (0,1) rectangle (1,2);
\draw[fill=blue!15] (0,0) rectangle (1,1);
\node at (0.5,-0.6) {$\vdots$};

% 标记
\fill[red] (1.5,2.5) circle (1.5pt);
\node[red,anchor=west] at (1.7,2.5) {$(2,\lambda_2)$};

% 虚线框出左侧/下方可用区域
\draw[blue,thick,dashed,rounded corners=3pt]
  (-0.2,-0.3) rectangle (1.2,2.8);
\node[blue,anchor=west] at (1.5,1.5) {$\leq t$ boxes};

% 文字
\node[anchor=west] at (9.5,2.5) {\textbf{Case $\lambda_2=2$}};
\node[anchor=west] at (9.5,1.6) {Rim hook through $(2,\lambda_2)$};
\node[anchor=west] at (9.5,0.7) {needs length $\geq t+1$,};
\node[anchor=west] at (9.5,-0.2) {but room on left/below is $\leq t$.};
\end{tikzpicture}

\vspace{1.5em}

% ============================================================
% 图 2：\lambda_2>2 的情形
% ============================================================
\begin{tikzpicture}[scale=0.55,every node/.style={font=\small}]

% --- 第一行 (最上方，长度 n-t) ---
\foreach \x in {0,1,2,3,4} {\draw[fill=gray!15] (\x,3) rectangle (\x+1,4);}
\draw[fill=gray!15] (5.5,3) rectangle (8.5,4);
\node at (7,3.5) {$\dots$};

% --- 第二行 (中间，长度 \geq 3，这里画 4 格) ---
\foreach \x in {0,1,2,3} {\draw[fill=blue!15] (\x,2) rectangle (\x+1,3);}

% --- 第 3 行及以下 (首列) ---
\draw[fill=blue!15] (0,1) rectangle (1,2);
\draw[fill=blue!15] (0,0) rectangle (1,1);
\node at (0.5,-0.6) {$\vdots$};

% 高亮 2\times 2 方块 (列 \lambda_2-2,\lambda_2-1)
\draw[red,very thick,fill=red!20,opacity=0.5]
  (1,2) rectangle (3,4);
\node[red] at (2,3) {$2\times2$};

% 标注
\draw[blue,thick,dashed,rounded corners=3pt]
  (-0.2,-0.3) rectangle (1.2,1.8);
\node[blue,anchor=west] at (1.5,0.8) {$\leq 2t$ boxes};

% 文字
\node[anchor=west] at (9.5,2.5) {\textbf{Case $\lambda_2>2$}};
\node[anchor=west] at (9.5,1.6) {$2\times2$ square forces};
\node[anchor=west] at (9.5,0.7) {two rim hooks ($\geq t+1$ each),};
\node[anchor=west] at (9.5,-0.2) {but total room is $\leq 2t$.};
\end{tikzpicture}
\label{graph1}
\end{figure}
Suppose that such a rim hook tableau $T$ exists. Since every cycle of $C$ has length at least $k$ and $t\le k-1$, every rim hook appearing in $T$ must have length at least $t+1$.

Consider the rim hook passing through the box $(2,\lambda_2)$. From the shape of $\lambda$, there are at most $t-1$ boxes lying either below or to the left of $(2,\lambda_2)$. Hence this rim hook must also pass through the box $(1,\lambda_2)$.

If $\lambda_2=2$, then there are at most $t$ boxes lying on the left or below $(2,\lambda_2)$ (see the first diagram Figure \ref{graph1}). This contradicts the requirement that the rim hook passing through $(1,1)$ must have length at least $t+1$.

Now suppose that $\lambda_2>2$. In this case the four boxes
$$
(1,\lambda_2-1),\ (1,\lambda_2-2),\ (2,\lambda_2-1),\ (2,\lambda_2-2)
$$
form a $2\times2$ square strictly to the left of the $\lambda_2$-th column (see the second diagram Figure \ref{graph1}). Consequently there exist at least two rim hooks whose boxes lie entirely either to the left of or below the $\lambda_2$-th column. Since each of these rim hooks has length at least $t+1$, they require at least
$
2(t+1)=2t+2
$
boxes in total. However, the number of boxes strictly to the left of the $\lambda_2$-th column is at most $2t$, which yields a contradiction.
Therefore no such rim hook tableau exists. $\square$

\begin{lemma}\label{chi1t}
Let $C$ be a conjugacy class of ${\mathfrak S}_n$ such that every cycle in $C$ has length at least $k$. Then for any $t\le k-1$,
\[
\chi^{(n-t,1^t)}(C)=(-1)^t .
\]
\end{lemma}

\begin{proof}
Let $\alpha=(\alpha_1,\ldots,\alpha_r)$ be the cycle type of $C$. Since $\alpha_i\ge k>t$ for all $i$, every rim hook of length $\alpha_i$ in the Young diagram of $\lambda=(n-t,1^t)$ must intersect the first row. Hence there is exactly one sequence of rim hook removals compatible with the Murnaghan--Nakayama rule: the first rim hook removes all $t$ boxes in the first column together with $\alpha_1-t$ boxes in the first row, and the remaining rim hooks lie entirely in the first row.

Thus there is exactly one rim hook tableau of shape $\lambda$ and content $\alpha$. Its sign is determined by the leg lengths of the rim hooks: the first rim hook has leg length $t$, while all subsequent rim hooks have leg length $0$. Therefore the total sign is $(-1)^t$. By the Murnaghan--Nakayama rule,
\[
\chi^{(n-t,1^t)}(C)=(-1)^t.
\]
\end{proof}

We introduce a formula for computing the exterior power $\wedge^t V$ of the standard representation $V$ of the symmetric group ${\mathfrak S}_n$. The character of $V$ is $\chi_V=\chi^{(n-1,1)}$. We write $\chi=\chi^{(n-1,1)}$ for brevity.
The exterior power $\wedge^t V$ is an irreducible representation corresponding to the partition $(n-t,1,\ldots,1)$ (see \cite{GTM129}, Exercise~4.6). As before, we denote this partition by $(n-t,1^t)$.

\begin{lemma}\label{wedgeformula}
Let $\sigma$ be a permutation in ${\mathfrak S}_n$, and let $\chi=\chi^{(n-1,1)}$. Then
\begin{equation}\label{wedgechar} \chi^{(n-t,1^t)}(\sigma)= \frac{1}{t!} \det \begin{pmatrix} \chi(\sigma) & 1 & & & & &\\ \chi (\sigma^2) & \chi(\sigma) & 2 & & & & \\ \chi (\sigma^3) & \chi(\sigma^2) & \chi(\sigma) & 3 & & & \\ \vdots & \vdots & \vdots & \vdots & \ddots & & \\ \chi (\sigma^{t-1}) & \chi (\sigma^{t-2}) & \cdots & \cdots & \cdots & t-1 \\ \chi (\sigma^{t}) & \chi (\sigma^{t-1}) & \cdots & \cdots & \cdots & \chi(\sigma) \end{pmatrix}. 
\end{equation}
\end{lemma}

\begin{proof}
The fundamental theorem of the representation theory of symmetric groups (see \cite{Knutson}, p.~124) states that the representation ring $R$ of all isomorphism classes of representations of symmetric groups is isomorphic to the ring of symmetric functions $\Lambda$ as $\lambda$-rings.

Under this correspondence, the $\lambda$-structure on $R$ is given by exterior powers, and the Adams operator satisfies
$$
\Psi^t(\chi^\lambda)(\sigma)=\chi^\lambda(\sigma^t)
$$
for any partition $\lambda$. On the ring of symmetric functions $\Lambda$, the $\lambda$-structure corresponds to the elementary symmetric functions $e_t$, while the Adams operator corresponds to the power-sum symmetric functions $p_t$.

The formula (\ref{wedgechar}) then follows from the closed form of Newton's identities relating $e_t$ and $p_t$ (see \cite{MacDonald}, I.2, Example~8). 
\end{proof}

\subsection{The main layers}\label{section4o1}

In this subsection we fix two  positive integer sequences $\{m_{1,n}\}_{n\in {\mathbb N}}$ and $\{m_{2,n}\}_{n\in {\mathbb N}}$ such that $\lim_{n\rightarrow \infty} m_{1,n}/ \sqrt{n}=H_1>0$ and $\lim_{n\rightarrow \infty} m_{1,n}/ \sqrt{n}=H_2>0$. We assume that $C_1,C_2,C_3 $ is a triple of conjugacy classes of ${\frak S}_n$ satisfying the conditions 1-4 in Theorem \ref{SimiGaussian}. Let ${\bf Z}_n^t={\bf Z}_n^t(C_1,C_2,C_3)$.

\begin{lemma}\label{SG-main-exact}
For all sufficiently large $n$,
\[
 \sum_{0\leq t<Q_n}{\bf Z}_n^t
 =
 \sum_{0\leq t<Q_n} q_n(t),
\]
where the sums are over integers $t$ and
\[
 q_n(t)=(-1)^t\frac{
 \binom{m_{1,n}-1}{t}\binom{m_{2,n}-1}{t}}{\binom{n-1}{t}}=(-1)^t\frac{1}{t!}
 \frac{\prod_{j=0}^{t-1}(m_{1,n}-1-j)(m_{2,n}-1-j)}
 {\prod_{j=0}^{t-1}(n-1-j)}.
\]
while $q_n(t)=0$ for $t\geq \min\{m_{1,n},m_{2,n}\}$.
\end{lemma}

\begin{proof}
By the sign condition and by
\(
 \chi^{\lambda'}(C)=\varepsilon(C)\chi^\lambda(C),
\)
the contribution of ${\mathscr Z}_n^t$ equals the contribution of
${\mathscr Y}_n^t$.  Thus it is enough to compute the sum over
${\mathscr Y}_n^t$.

Let $0\leq t<Q_n$.  Since $C_3$ has no cycles of integer length
$i<Q_n$, every cycle of an element of $C_3$ has length greater than $t$.
By Lemma~\ref{important}, among diagrams in ${\mathscr Y}_n^t$, the only
one which may have a nonzero character value at $C_3$ is
\(
 (n-t,1^t).
\)
Moreover, by Lemma~\ref{chi1t},
\(
 \chi^{(n-t,1^t)}(C_3)=(-1)^t.
\)

Let $C$ be a conjugacy class which
has exactly $m$ fixed points and has no cycles of integer length
$2\leq i<Q_n$.  Hence, for every integer $1\leq s\leq t$, the permutation
$\sigma^s$ has the same fixed points as $\sigma$.  Therefore
\(
 \chi^{(n-1,1)}(\sigma^s)=m-1
 \) for \(1\leq s\leq t\).
Using the Newton identity formula for exterior powers, Lemma~\ref{wedgeformula},
we obtain
\[
 \chi^{(n-t,1^t)}(C)=\binom{m-1}{t}.
\]
Moreover,
\[
 \chi^{(n-t,1^t)}({\bf 1})=\binom{n-1}{t}.
\]
When $C$ becomes $C_1$ or $C_2$, it will have either $m_{1,n}$ or $m_{2,n}$ fixed points. Thus the $t$-th term contributes
\[
 (-1)^t\frac{
 \binom{m_{1,n}-1}{t}\binom{m_{2,n}-1}{t}}{\binom{n-1}{t}}=q_n(t).
\]
If $t\geq \min\{m_{1,n},m_{2,n}\}$, then either $\binom{m_{1,n}-1}{t}=0$ or $\binom{m_{2,n}-1}{t}=0$, so $q_n(t)=0$.
Expanding the binomial expression
for $q_n(t)$ gives
\[
 \frac{\binom{m_{1,n}-1}{t}\binom{m_{2,n}-1}{t}}{\binom{n-1}{t}}
 =
 \frac{1}{t!}
 \frac{\prod_{j=0}^{t-1}(m_{1,n}-1-j)(m_{2,n}-1-j)}
 {\prod_{j=0}^{t-1}(n-1-j)}.
\]

\end{proof}

\begin{lemma}\label{SG-An-uniform} Notations as in Lemma \ref{SG-main-exact}.  Let
\(
 r_n=\lfloor n^{2/9}\rfloor,
\) then 
for $0\leq t\leq r_n$,
\[
 q_n(t)=\frac{(-1)^t}{t!}\mu_n^t A_n(t).
\]
where
\[
 \mu_n=\frac{(m_{1,n}-1)(m_{2,n}-1)}{n-1},
\]
and, for $t\geq0$,
\[
 A_n(t)=\prod_{j=0}^{t-1}
 \frac{\left(1-\frac{j}{m_{1,n}-1}\right)\left(1-\frac{j}{m_{2,n}-1}\right)}
 {1-\frac{j}{n-1}},
\]
with the convention $A_n(0)=1$. Moreover, 
 for every $\eta>0$, there exists $N=N(\eta,H_1,H_2)$ such that, for all
$n>N$ and all integers $0\leq t\leq r_n$,
\[
 |A_n(t)-1|<\eta.
\]

\end{lemma}

\begin{proof} 
Factoring out $(m_{1,n}-1-j)(m_{2,n}-1-j)$ from each numerator factor and $n-1$ from each
denominator factor gives
\[
 \frac{1}{t!}
 \frac{\prod_{j=0}^{t-1}(m_{1,n}-1-j)(m_{2,n}-1-j)}
 {\prod_{j=0}^{t-1}(n-1-j)}
 =
 \frac{1}{t!}\mu_n^t A_n(t),
\]
which proves the displayed identity for $q_n(t)$.

It remains to prove the uniform estimate for $A_n(t)$.  For all sufficiently
large $n$ and all $0\leq j\leq r_n$, we have
\[
 0\leq \frac{j}{m_{1,n}-1}\leq\frac12,
 \qquad
 0\leq \frac{j}{m_{2,n}-1}\leq\frac12,
 \qquad
 0\leq \frac{j}{n-1}\leq\frac12.
\]
Using
\[
 |\log(1-x)|\leq 2x\qquad(0\leq x\leq1/2),
\]
we get, for $0\leq t\leq r_n$,
\begin{align*}
 |\log A_n(t)|
 &\leq
 \sum_{j=0}^{t-1}\left|\log\left(1-\frac{j}{m_{1,n}-1}\right)\right|
 +\sum_{j=0}^{t-1}\left|\log\left(1-\frac{j}{m_{2,n}-1}\right)\right|
 +
 \sum_{j=0}^{t-1}\left|\log\left(1-\frac{j}{n-1}\right)\right| \\
 &\leq
 2\sum_{j=0}^{t-1}\frac{j}{m_{1,n}-1}
 +2\sum_{j=0}^{t-1}\frac{j}{m_{2,n}-1}
 +2\sum_{j=0}^{t-1}\frac{j}{n-1} \\
 &\leq
 \frac{t^2}{m_{1,n}-1}+\frac{t^2}{m_{2,n}-1}+\frac{t^2}{n-1}.
\end{align*}
Since $t\leq r_n=\lfloor n^{2/9}\rfloor$ and both $1/(m_{1,n}-1)$ and $1/(m_{2,n}-1)$ are $O(1/\sqrt{n})$,  the last expression is bounded by
\[
 C_H n^{-\frac{1}{18}}+Cn^{-\frac{5}{9}}.
\]  Hence, after increasing $N$ if
necessary,
\(
 |\log A_n(t)|<\eta/2
\)
for all $0\leq t\leq r_n$.  Then
\[
 |A_n(t)-1|\leq e^{|\log A_n(t)|}-1<\eta.
\]
\end{proof}

\begin{lemma}\label{SG-taylor-main}
For every $\eta>0$, there exists $N=N(\eta,H_1,H_2)$ such that, for all $n>N$,
\[
 \left|\sum_{0\leq t\leq r_n}q_n(t)-{\rm e}^{-H_1H_2}\right|<\eta.
\]
\end{lemma}

\begin{proof}
Choose $B>H_1H_2+1$, then  for all sufficiently large $n$, we have
\(
 0\leq \mu_n\leq B.
\)
By Lemma~\ref{SG-An-uniform}, after increasing $N$, we may also assume
\(
 0\leq A_n(t)\leq 2\) for 
 \(0\leq t\leq r_n.
\)
Therefore
\[
 |q_n(t)|\leq \frac{2B^t}{t!}
 \qquad(0\leq t\leq r_n).
\]

Let $\eta>0$.  Choose an integer $M$ large enough so that
\[
 2\sum_{t>M}\frac{B^t}{t!}
 +
 \sum_{t>M}\frac{(H_1H_2)^t}{t!}
 <\frac{\eta}{3}.
\]
For this fixed $M$, we have $\mu_n$ tend to $H_1H_2$, and
Lemma~\ref{SG-An-uniform} gives $A_n(t)\to1$ uniformly for $0\leq t\leq M$.
Hence, for all sufficiently large $n$,
\[
 \left|
 \sum_{t=0}^{M}\frac{(-1)^t}{t!}\mu_n^tA_n(t)
 -
 \sum_{t=0}^{M}\frac{(-1)^t}{t!}(H_1H_2)^t
 \right|<\frac{\eta}{3}.
\]
After increasing $N$ once more, we may assume $r_n>M$.  Using the Taylor expansion
we get
$$\left|\sum_{0\leq t\leq r_n}q_n(t)-{\rm e}^{-H_1H_2}\right|
 \leq
 \left|
 \sum_{t=0}^{M}q_n(t)
 -
 \sum_{t=0}^{M}\frac{(-1)^t(H_1H_2)^t}{t!}
 \right| 
 +
 \sum_{M<t\leq r_n}|q_n(t)|
 +
 \sum_{t>M}\frac{(H_1H_2)^t}{t!} 
 <\eta.$$
\end{proof}

\begin{lemma}\label{SG-middle-tail}
For every $\eta>0$, there exists $N=N(\eta,H_1,H_2,P)$ such that, for all $n>N$,
\[
 \left|\sum_{r_n<t<Q_n}q_n(t)\right|<\eta.
\]
\end{lemma}

\begin{proof}
If $t\geq \min\{m_{1,n},m_{2,n}\}$, then $q_n(t)=0$.  Thus it is enough to consider $t< \min\{m_{1,n},m_{2,n}\}$.
Since $Q_n/n=P\log n/\sqrt n\to0$, for all sufficiently large $n$ we have $Q_n<n/2$.  Hence, for
$0\leq j\leq t-1<Q_n$, we have
\(
 n-1-j\geq \frac n2.
\)
Since $m_{i,n}\sim H_i\sqrt{n}$, for $n$ large enough, we may assume that $m_{i,n}\leq 2H_i\sqrt{n}$ for $i=1,2$.
Using the product formula from Lemma~\ref{SG-main-exact}, we get
\begin{align*}
 |q_n(t)|
 =
  \frac{1}{t!}
 \frac{\prod_{j=0}^{t-1}(m_{1,n}-1-j)(m_{2,n}-1-j)}
 {\prod_{j=0}^{t-1}(n-1-j)}
 \leq
 \frac1{t! }(2H_1\sqrt n)^t(2H_2\sqrt n)^t\left(\frac2n\right)^t
 =
 \frac{(4H_1H_2)^t}{t!}.
\end{align*}
Using $t!\geq ((t-1)/{\rm e})^{t-1}$, we obtain
\[
 |q_n(t)|\leq \left(\frac{4{\rm e}H_1H_2}{t-1}\right)^{t-1}.
\]
Put $A=4{\rm e}H_1H_2$.  The function $(A/(t-1))^{t-1}$ is decreasing for all sufficiently
large $t$.  Since $r_n\to\infty$, after increasing $N$ we may assume it is
decreasing for $t\geq r_n-1 \geq n^{\frac{2}{9}}-1$.  Thus
\[
 \sum_{r_n<t<Q_n}|q_n(t)|
 \leq
 2Q_n\left(\frac{A}{r_n}\right)^{r_n} \leq 2 P \sqrt{n} \log n \left(\frac{A}{n^{\frac{2}{9}}-1}\right)^{n^{\frac{2}{9}}-1}
\] 
because the number of integers $t$ with $r_n<t<Q_n$ is at most
$\lceil Q_n\rceil\leq 2Q_n$ for all sufficiently large $n$. This product is dominated by the term $(n^{\frac{2}{9}}-1)^{-(n^{\frac{2}{9}}-1))}$ and thus tends to $0$.
\end{proof}

\begin{lemma}\label{SG-main-layers-estimate}
For every $\eta>0$, there exists $N=N(\eta,H_1,H_2,P)$ such that, for all $n>N$,
\[
 \left|\sum_{0\leq t<Q_n}{\bf Z}_n^t-{\rm e}^{-H_1H_2}\right|<\eta.
\]
\end{lemma}

\begin{proof}
By Lemma~\ref{SG-main-exact},
\[
 \sum_{0\leq t<Q_n}{\bf Z}_n^t
 =
 \sum_{0\leq t<Q_n}q_n(t).
\]
Split the latter sum as
\[
 \sum_{0\leq t<Q_n}q_n(t)
 =
 \sum_{0\leq t\leq r_n}q_n(t)
 +
 \sum_{r_n<t<Q_n}q_n(t).
\]
Apply Lemma~\ref{SG-taylor-main} with $\eta/2$ and
Lemma~\ref{SG-middle-tail} with $\eta/2$.
\end{proof}

\section{The Tail Approximation}\label{section4o2}

In this section we prove Proposition \ref{SG-X-tail-estimate} and Theorem \ref{SimiGaussian}. Recall the decomposition
\[
 \mathbf{Y}_n=2\sum_{0\le t<Q_n}\mathbf{Z}_n^t+\mathbf{X}_n^{Q_n},
\]
where the main layers have been shown to converge to ${\rm e}^{-H_1H_2}$ in the previous section. It remains to control the tail term $\mathbf{X}_n^{Q_n}$, which consists of those irreducible representations whose first row and first column are both short (bounded by $n-Q_n$). For these diagrams the character values at the given conjugacy classes are small compared with their dimensions, and we will exploit this through sharp lower bounds for the dimensions. The following two subsections establish these lower bounds and then apply the Larsen--Shalev character ratio estimate to show that the tail contribution tends to zero.

\subsection{Approximations of $\chi^{\lambda}({\bf 1})$}

\begin{lemma}\label{two-row-min-degree}
Let $ t<(n-1)/2$.  Then
\[
 \min\{\chi^{\lambda}({\bf 1})\mid \lambda\in{\mathscr Z}_n^t\}
 =\chi^{(n-t,t)}({\bf 1}).
\]
Moreover,
\[
 \chi^{(n-t,t)}({\bf 1})
 =\binom nt\frac{n+1-2t}{n+1-t}.
\]
\end{lemma}

\begin{proof}
Since conjugation preserves dimensions, it is enough to prove the corresponding statement for ${\mathscr Y}_n^t$.  Let
\(
 \lambda=(n-t,\mu),
  \mu\in {\mathscr Y}_t.
\)
The assumption $t<(n-1)/2$ implies $n-t>t$, so the first row is strictly longer than every row of $\mu$.

Let $H(\nu)$ denote the product of hook lengths of a Young diagram $\nu$.
By the hook-length formula, minimizing $\chi^{\lambda}({\bf 1})=n!/H(\lambda)$ is equivalent to maximizing $H(\lambda)$.
Write the conjugate of $\mu$ as
\(
 \mu'=(c_1,c_2,\ldots,c_t),
\)
where zeros are appended if necessary, so that
\[
 c_1\geq c_2\geq\cdots\geq c_t\geq0,
 \qquad
 c_1+\cdots+c_t=t.
\]
For the boxes below the first row, the hook lengths are exactly the hook lengths of $\mu$.  For the first row, the hook lengths in columns $1,\ldots,t$ are
\[
 n-t-j+1+c_j \qquad (1\leq j\leq t),
\]
and the hook lengths in columns $t+1,
\ldots,n-t$ have product $(n-2t)!$.  Hence
\[
 H(\lambda)
 =H(\mu)(n-2t)!J(c),
\]
where
\[
 J(c)=\prod_{j=1}^t(n-t-j+1+c_j).
\]

We first note that
\[
 H(\mu)=\frac{t!}{\chi^\mu({\bf 1})}\leq t!,
\]
with equality, for example, for the one-row diagram $(t)$.  It remains to show that $J(c)$ is maximized when
\(
 c=(1,1,\ldots,1),
\)
which corresponds to $\mu=(t)$.
Suppose $c\neq(1,\ldots,1)$.  Since the entries are non-negative integers with sum $t$, there exist indices $i<j$ such that
\[
 c_i\geq c_j+2.
\]
Choose such a pair as follows.  First take a pair $p<q$ with $c_p\geq c_q+2$; then replace $p$ by the last index in the constant block containing $p$, and replace $q$ by the first index in the constant block containing $q$.  The resulting indices, denoted by $i<j$, still satisfy $c_i\geq c_j+2$, and decreasing $c_i$ by $1$ while increasing $c_j$ by $1$ preserves the non-increasing order (see Figure \ref{balancing} as  an example).  Let $\widetilde c$ be the resulting sequence.  Then
\[
 \widetilde c_i=c_i-1,
 \qquad
 \widetilde c_j=c_j+1,
 \qquad
 \widetilde c_l=c_l\quad(l\neq i,j).
\]

\begin{figure}[H]\label{balancing}
\centering
\begin{tikzpicture}[scale=0.72,every node/.style={font=\small}]

% ==================== 左图 ====================
\node at (3,5.2) {$\lambda=(n-4,3,1),\quad c=(2,1,1,0)$};

% 第3行（最下，1格）
\draw[fill=blue!55] (0,0) rectangle (1,1);
% 第2行（3格）
\foreach \x in {0,1,2} {\draw[fill=blue!55] (\x,1) rectangle (\x+1,2);}
% 第1行（n-4格，灰色）
\foreach \x in {0,1,2,3,4,5} {\draw[fill=gray!35] (\x,2) rectangle (\x+1,3);}
\node at (6.3,2.5) {$\cdots$};

% 高亮第1列 c_1=2（粉红覆盖第2、3行）
\fill[red!25,opacity=0.7] (-0.08,0) rectangle (1.08,2);
% 高亮第4列（标记j=4）
\fill[red!25,opacity=0.7] (2.92,2) rectangle (4.08,3);

% 底部列高度 c_j
\foreach \x/\val in {0/2,1/1,2/1,3/0} {
  \draw[->,thick,red!70!black] (\x+0.5,0.5) -- (\x+0.5,0.9);
  \node[red!70!black,below] at (\x+0.5,-0.05) {$\val$};
}
\node at (1.5,-1.1) {column heights $c_j$};

% ==================== 中间操作标注 ====================
% 分数线式标注
\node[red!70!black] at (9,3.8) {$c_1=2$};
\draw[thick] (8.5,3.35) -- (10.5,3.35);
\node[red!70!black] at (9,2.9) {$c_4=0$};

% 映射说明
\node[font=\footnotesize] at (9,2.2) {$c_1\mapsto c_1-1$};
\node[font=\footnotesize] at (9,1.6) {$c_4\mapsto c_4+1$};
\node[font=\footnotesize] at (9,1) {$c_1\geq c_4+2$};

% 粗箭头
\draw[->,line width=1.2pt] (7,2.5) -- (11.5,2.5);

% ==================== 右图 ====================
\node at (16,5.2) {$\lambda=(n-4,4),\quad \widetilde c=(1,1,1,1)$};

% 第2行（4格，绿色）
\foreach \x in {0,1,2,3} {\draw[fill=green!55] (\x+13,1) rectangle (\x+14,2);}
% 第1行（n-4格，灰色）
\foreach \x in {0,1,2,3,4,5} {\draw[fill=gray!35] (\x+13,2) rectangle (\x+14,3);}
\node at (19.3,2.5) {$\cdots$};

% 底部列高度 \widetilde c_j
\foreach \x/\val in {0/1,1/1,2/1,3/1} {
  \draw[->,thick,green!50!black] (\x+13.5,0.5) -- (\x+13.5,0.9);
  \node[green!50!black,below] at (\x+13.5,-0.05) {$\val$};
}
\node at (15.5,-1.1) {$\widetilde c_j$};

\end{tikzpicture}
\caption{The balancing operation for $t=4$.
Left:~$\lambda=(n-4,3,1)$ has column heights $c=(2,1,1,0)$.
Right:~after replacing $c_1\leftarrow c_1-1$ and
$c_4\leftarrow c_4+1$, the new diagram is $(n-4,4)$ with
$\widetilde c=(1,1,1,1)$.  This operation preserves the
non-increasing order and strictly increases $J(c)$.}

\end{figure}

All factors in $J(c)$ except the $i$-th and $j$-th remain unchanged.  Put
\[
 A_i=n-t-i+1,
 \qquad
 A_j=n-t-j+1.
\]
Then
\begin{align*}
 &(A_i+c_i-1)(A_j+c_j+1)-(A_i+c_i)(A_j+c_j) \\
 &=(A_i+c_i)-(A_j+c_j)-1 =(j-i)+(c_i-c_j)-1>0.
\end{align*}
Thus $J(\widetilde c)>J(c)$.  Repeating this  operation strictly increases $J$ until all entries differ by at most $1$.  Since their sum is $t$ and there are $t$ entries, the terminal sequence is $(1,\ldots,1)$.  Therefore
\(
 J(c)\leq J(1,\ldots,1).
\)

Combining the two estimates gives
\[
 H(\lambda)
 \leq t!(n-2t)!J(1,\ldots,1).
\]
The right-hand side is exactly the hook-length product of the two-row diagram $(n-t,t)$.  Hence $H(\lambda)\leq H(n-t,t)$, and so
\[
 \chi^{\lambda}({\bf 1})\geq \chi^{(n-t,t)}({\bf 1}).
\]

It remains to compute the latter dimension.  For the diagram $(n-t,t)$, the hook-length product is
\(
 \frac{(n-t+1)!\,t!}{n-2t+1}.
\)
Thus the hook-length formula gives
\[
 \chi^{(n-t,t)}({\bf 1})
 =\frac{n!(n-2t+1)}{(n-t+1)!t!}
 =\binom nt\frac{n+1-2t}{n+1-t}.
\]
\end{proof}

The next result is essentially Theorem 2 in \cite{Mishchenko1996}, and this proof uses the same idea as in \cite{Mishchenko1996}. The difference is that we need a uniform lower bound for all Young diagrams contained in the top-left square of size $(n/k)\times (n/k)$, while in \cite{Mishchenko1996} the result is stated for a given sequence.

\begin{lemma}\label{uniform-exponential-lower-bound}
Let $k\ge 1$ be fixed. For every real number $B$ with $0<B<k$, there exists an integer $N=N(k,B)$ such that for any $n\ge N$ and $\lambda\in {\mathscr Y}_n$ is a Young diagram satisfying
\[
   \lambda_1\le \frac{n}{k},\qquad \lambda'_1\le \frac{n}{k},
\]
then
\[
   \chi^{\lambda}({\bf 1})>B^n,
\]
where $\chi^{\lambda}({\bf 1})$ denotes the dimension of the irreducible $S_n$-module corresponding to $\lambda$.
\end{lemma}

\begin{proof}
We may assume $1<B<k.$
Put
\(
   \gamma=\log k-\log B>0.
\)
We shall prove the assertion uniformly for all diagrams satisfying the two displayed inequalities. The constants denoted by $C_0,C_1,\ldots$ below are positive constants; in particular they do not depend on $n$, $\lambda$, $k$, or $B$.

Let $m$ be the side length of the {\bf Durfee square} of $\lambda$, i.e. the largest integer such that the square $(m^m)$ is contained in $\lambda$. Write the Frobenius coordinates of $\lambda$ as
\(
   a_i=\lambda_i-i,
   b_i=\lambda'_i-i,
 i=1,\ldots,m.
\)
Then \(
   a_1>a_2>\cdots>a_m\ge 0,
   b_1>b_2>\cdots>b_m\ge 0,
\)
and
\(
   n=\sum_{i=1}^m (a_i+b_i+1).
\)
Moreover, since $\lambda_1\le n/k$ and $\lambda'_1\le n/k$, we have
\[
   a_i\le \lambda_i\le \lambda_1\le \frac{n}{k},
   \qquad
   b_i\le \lambda'_i\le \lambda'_1\le \frac{n}{k}.
\]
After increasing the final value of $N$ if necessary, we may and shall also assume that $n\ge k$, so that $1\le n/k$.
The hook formula in Frobenius coordinates (see \cite[page 14, Proposition 1.3]{Borodin}) gives
\begin{equation}\label{prodchi}
   \chi^{\lambda}({\bf 1})
   =
   \frac{
      n!
   }{
      \prod_{i=1}^m a_i!b_i!
   }\cdot \frac{\prod_{1\le i<j\le m}(a_i-a_j)(b_i-b_j)}{\prod_{i,j=1}^m(a_i+b_j+1)}.
\end{equation}
We will proceed the proof by several steps.

{\em Step 1: The multinomial approximation}\

 We first estimate the multinomial factor in (\ref{prodchi}). Consider the $3m$ non-negative integers
\[
   a_1,\ldots,a_m,
   \quad
   b_1,\ldots,b_m,
   \quad
   \underbrace{1,\ldots,1}_{m\text{ times}}.
\]
Denote them by $x_1,\ldots,x_{3m}$. Their sum is $n$, and each of them is at most $n/k$. Set $p_\alpha=x_\alpha/n$. Then
\[
   p_\alpha\ge 0,
   \qquad
   \sum_{\alpha=1}^{3m}p_\alpha=1,
   \qquad
   p_\alpha\le \frac1k.
\]
With the convention $0\log 0=0$, we have
\[
   \sum_{\alpha=1}^{3m}p_\alpha\log\frac1{p_\alpha}
   \ge
   \sum_{\alpha=1}^{3m}p_\alpha\log k
   =\log k.
\]
We use the elementary form of Stirling's estimate
\[
   s\log s-s\le \log(s!)\le s\log s-s+C_0\log(s+2)
   \qquad (s=0,1,2,\ldots),
\]
again with the convention $0\log0=0$. Applying the lower bound to $n!$ and the upper bound to each $x_\alpha!$, and using $x_\alpha\le n$, gives
\begin{align*}
   &\log\frac{n!}{\prod_{i=1}^m a_i!b_i!}
   =
   \log\frac{n!}{\prod_{\alpha=1}^{3m}x_\alpha!} \ge
   n\log n-n
   -\sum_{\alpha=1}^{3m}(x_\alpha\log x_\alpha-x_\alpha)
   -3C_0m\log(n+2) \\
   & =
   n\sum_{\alpha=1}^{3m}p_\alpha\log\frac1{p_\alpha}
   -3C_0m\log(n+2) \ge
   n\log k-C_1m\log(n+2).
\end{align*}
The extra factors $1!$ do not change the product of factorials, and the convention $0!=1$ handles the possible zero values among the $a_i$ and $b_i$.

{\em Step 2: The remaining factor approximation}\

For the second factor
\[
R = \frac{\prod_{1\le i<j\le m}(a_i-a_j)(b_i-b_j)}{\prod_{i,j=1}^m(a_i+b_j+1)}
\]
of $\chi^\lambda(\mathbf{1})$ in \eqref{prodchi}, set $q_i=a_i+b_i+1$.
Then $q_1>q_2>\cdots>q_m\ge1$, $\sum_i q_i=n$, and $q_j\le n/j$.
Rewriting $R$ as
\[
R = \prod_{i=1}^m q_i \prod_{1\le i<j\le m}
\frac{(a_i+b_j+1)(a_j+b_i+1)}{(a_i-a_j)(b_i-b_j)},
\]
and using $a_i-a_j\ge j-i$, $b_i-b_j\ge j-i$ for $i<j$, we obtain
\[
\frac{(a_i+b_j+1)(a_j+b_i+1)}{(a_i-a_j)(b_i-b_j)}
\le \left(1+\frac{q_j}{j-i}\right)^2.
\]
Hence for each $j$,
\[
\prod_{i=1}^{j-1}\left(1+\frac{q_j}{j-i}\right)
= \prod_{d=1}^{j-1}\left(1+\frac{q_j}{d}\right)
\le \binom{q_j+j}{j}.
\]
Therefore
\[
R \le n^m \prod_{j=1}^m \binom{q_j+j}{j}^2.
\]
By the standard bound $\binom{u}{v}\le (\mathrm{e}u/v)^v$,
\[
\binom{q_j+j}{j} \le \left(\mathrm{e}\Bigl(1+\frac{q_j}{j}\Bigr)\right)^{\!j}.
\]
Combining with $q_j\le n/j$ and $j\le m\le\sqrt{n}$, we get
\[
\log R \le m\log n + 2\sum_{j=1}^m j + 2\sum_{j=1}^m j\log\Bigl(1+\frac{n}{j^2}\Bigr).
\]
The last sum is bounded by splitting at $j\le m$ and using
$\log(1+x)\le \log(2x)$ for $x\ge1$:
\[
\sum_{j=1}^m j\log\Bigl(1+\frac{n}{j^2}\Bigr)
\le \sum_{j=1}^m j\log\frac{2n}{j^2}
\le m^2\log\frac{2n}{m^2} + 2\sum_{j=1}^m j\log\frac{m}{j}.
\]
Since $\sum_{j=1}^m j\log(m/j) \le m^2/\mathrm{e}$ (by $x\log(1/x)\le 1/\mathrm{e}$),
we finally obtain
\[
\log R \le C m\log n + C m^2\log\frac{\mathrm{e}n}{m^2}
\]
with an absolute constant $C$.

{\em Step 3: The approximation for the dimension}\

Combining the estimate for $R$ with the estimate of the multinomial factor, we obtain
\begin{equation}\label{eq:uniform-main-estimate}
   \log \chi^{\lambda}({\bf 1})
   \ge
   n\log k
   -C_4m\log(n+2)
   -C_4m^2\log\frac{{\rm e}n}{m^2}.
\end{equation}

We now choose a number $\delta\in(0,1)$, depending only on $k$ and $B$, such that
\[
   C_4\delta\log\frac{\rm e}{\delta}<\frac{\gamma}{3}.
\]
This is possible because $\delta\log({\rm e}/\delta)\to0$ as $\delta\downarrow0$.

{\em Step 4: The small Durfee square case}\

Assume first that
\(
   m^2\le \delta n.
\)
The function $x\mapsto x\log(en/x)$ is increasing for $0<x\le n$, since its derivative is $\log(n/x)$. Therefore
\(
   m^2\log\frac{en}{m^2}
   \le
   \delta n\log\frac{e}{\delta}.
\)
Moreover $m\le \sqrt{\delta n}$, so
\(
   C_4m\log(n+2)
   \le
   C_4\sqrt{\delta n}\log(n+2)=o(n).
\)
Thus, increasing $N$ if necessary, we may assume that for all $n\ge N$,
\(
   C_4\sqrt{\delta n}\log(n+2)<\frac{\gamma}{3}n.
\)
Substituting these estimates into \eqref{eq:uniform-main-estimate}, we get
\[
   \log \chi^{\lambda}({\bf 1})
   \ge
   n\log k-\frac{\gamma}{3}n-\frac{\gamma}{3}n
   >
   n\log B.
\]
Hence $\chi^{\lambda}({\bf 1})>B^n$ in this case.

{\em Step 5: The big Durfee square case}\

It remains to consider the case
\(
   m^2>\delta n.
\)
Since the square diagram $(m^m)$ is contained in $\lambda$, we have
\[
   \chi^{\lambda}({\bf 1})\ge \chi^{(m^m)}({\bf 1}).
\]
Indeed, if $\mu\subset\lambda$ is any subdiagram, then $\chi^{\lambda}({\bf 1})\ge f^\mu$. Recall that $\chi^{\lambda}({\bf 1})$ also equals to the number of those standard Young tableaux with shape $\lambda$ (see \cite[Section 2.5]{GTM203}). Fix a standard filling of the skew diagram $\lambda/\mu$ with the numbers $|\mu|+1,\ldots,|\lambda|$ (for instance, fill the rows from top to bottom and, within each row, from left to right), and concatenate it with any standard tableau of shape $\mu$. Since all entries in the skew part are larger than all entries in $\mu$, this gives an injection from standard tableaux of shape $\mu$ to standard tableaux of shape $\lambda$.

For the square diagram $(m^m)$, every hook length is at most $2m$. Therefore the hook formula and using the elementary Stirling lower bound $r!\ge (r/e)^r$ with $r=m^2$, we obtain
\[
   \chi^{(m^m)}({\bf 1})
   =
   \frac{(m^2)!}{\prod_{(i,j)\in(m^m)}h_{ij}}
   \ge
   \frac{(m^2)!}{(2m)^{m^2}}
   \ge
   \left(\frac{m^2}{e}\right)^{m^2}\frac{1}{(2m)^{m^2}}
   =
   \left(\frac{m}{2e}\right)^{m^2}.
\]
Since $m^2>\delta n$, we have $m>\sqrt{\delta n}$. Choose $N$ still larger, if necessary, so that for every $n\ge N$,
\[
   \log\left(\frac{\sqrt{\delta n}}{2e}\right)
   >
   \frac{\log B}{\delta}.
\]
Then
\[
   \log \chi^{\lambda}({\bf 1})
   \ge
   \log f^{(m^m)}
   \ge
   m^2\log\frac{m}{2e}
   >
   \delta n\cdot \frac{\log B}{\delta}
   =
   n\log B.
\]
Thus $\chi^{\lambda}({\bf 1})>B^n$ also in the second case.

Combining the two cases, and taking $N$ to be the maximum of all lower bounds imposed above, we obtain an integer $N=N(k,B)$ such that for every $n\ge N$ and every Young diagram $\lambda\vdash n$ satisfying $\lambda_1,\lambda'_1\le n/k$, one has $\chi^{\lambda}({\bf 1})>B^n$. This proves the lemma.
\end{proof}

\subsection{Proof of the Semi-Gaussian Law}

We come back to the Proof of Theorem \ref{SimiGaussian}. Again we fix positive integer sequences $(m_{i,n})$  such that $\lim_{n\rightarrow \infty} m_{i,n}/ \sqrt{n}=H_i$ for $i=1,2$, and assume that $C_1,C_2,C_3 $ is a triple of conjugacy classes satisfying the conditions 1-4 in Theorem \ref{SimiGaussian}.

\begin{lemma}\label{SG-centralizer1}
Let $C$ be a conjugacy class of ${\mathfrak S}_n$.  Suppose that elements of
$C$ have exactly $m$ fixed points and that all other cycles have length at
least the real threshold $Q_n$.  Let
\(
 z(C)=|Z_{{\mathfrak S}_n}(\sigma)|,
  \sigma\in C.
\)
Here
 $Z_{{\mathfrak S}_n}(\sigma)=\{\tau\in{\mathfrak S}_n\mid \tau\sigma=\sigma\tau\}
$
is the centralizer of \(\sigma\) in \({\mathfrak S}_n\).
If $r$ is the number of non-fixed cycles of $\sigma$, then
\[
 z(C)\leq \Gamma(m+1)n^r\Gamma(r+1),
 \qquad \mbox{and}\qquad 
 r\leq \frac{n-m}{Q_n}.
\]
\end{lemma}

\begin{proof}
Write the cycle type of $C$ as
\(
 (1^m,a_1^{m_{a_1}},a_2^{m_{a_2}},\ldots),
\)
where all $a_i\geq Q_n$.  Then
\[
 z(C)=m!\prod_a a^{m_a}m_a!
 =\Gamma(m+1)\prod_a a^{m_a}\Gamma(m_a+1).
\]
Let
\(
 r=\sum_a m_a.
\)
Since each non-fixed cycle has length at least $Q_n$,
\(
 rQ_n\leq n-m,
\)
and hence
\[
 r\leq\frac{n-m}{Q_n}.
\]
Also $a\leq n$ and
\[
 \prod_a \Gamma(m_a+1)=\prod_a m_a!
 \leq \left(\sum_a m_a\right)!=\Gamma(r+1).
\]
Thus
\[
 z(C)\leq \Gamma(m+1)n^r\Gamma(r+1).
\]
\end{proof}

\begin{lemma}\label{SG-centralizer}
Fix real numbers $H_1>0$, $H_2>0$ and $P>0$. Let $\{m_{1,n}\}_{n \in {\mathbb N}}$ and $\{m_{2,n}\}_{n \in {\mathbb N}}$ be two sequences of non-negative integers such that 
\(
\lim_{n\rightarrow \infty}\frac{m_{1,n}}{\sqrt{n}}=H_1, \lim_{n\rightarrow \infty}\frac{m_{2,n}}{\sqrt{n}}=H_2.
\)
Let $Q_n=P\sqrt{n} \log n$. Then for $(H_1,H_2,P)$, and for $n$ large enough,
\begin{align*}
 &\log z(C_1)
 \leq
 H_1\sqrt n\log n+O_{H_1,P}(\sqrt n), \\
 &\log z(C_2)
 \leq
 H_2\sqrt n\log n+O_{H_2,P}(\sqrt n),
\end{align*}
and
\[
 \log z(C_3)=O_P(\sqrt n).
\]
\end{lemma}

\begin{proof}
We now record the size of the two factors involving the number $r$ in Lemma \ref{SG-centralizer1}.  Since
\(
 Q_n=P\sqrt n\log n,
\)
we have
\[
 r\leq \frac{n}{Q_n}=\frac1P\frac{\sqrt n}{\log n}.
\]
It follows immediately that
\[
 r\log n\leq \frac1P\sqrt n=O_P(\sqrt n).
\]
Moreover, if $r=0$, then $\log\Gamma(r+1)=0$.  If $r\geq1$, then $r$ is an
integer and $\Gamma(r+1)=r!$.  Using $r!\leq r^r$, we get
\[
 \log\Gamma(r+1)=\log(r!)\leq r\log r\leq r\log n\leq O_P(\sqrt n).
\]

For $C_1$ and $C_2$, we only proof the statement for $C_1$ because the same proof works for $C_2$. We have $\lim_{n\rightarrow \infty} m_{1,n}/ \sqrt{n}=H_1$. Hence for $n$ large enough, we assume that $m_{1,n} \leq 2H_1\sqrt{n}$.   Therefore, according to Lemma \ref{SG-centralizer1},
\[
 \log z(C_1)
 \leq
 \log\Gamma(2H_1\sqrt n+1)+r\log n+\log\Gamma(r+1)
 \qquad(i=1,2).
\]
By Stirling's formula,
\[
 \log\Gamma(2H_1\sqrt n+1)
 =2H_1\sqrt n\log(2H_1\sqrt n)-2H_1\sqrt n+O(\log n)
 =H_1\sqrt n\log n+O_{H_1}(\sqrt n).
\]
Together with the estimates above, this gives
\[
 \log z(C_1)
 \leq
 H_1\sqrt n\log n+O_{H_1,P}(\sqrt n).
\]
For $C_3$, we have $m=0$, and hence
\[
 \log z(C_3)\leq r\log n+\log\Gamma(r+1)=O_P(\sqrt n).
\]
\end{proof}

\begin{lemma}\label{SG-degree-lower}
There exists $N=N(P)$ such that, for all $n>N$ and all
$\lambda\in{\mathscr X}_n^{Q_n}$,
\[
 \chi^{\lambda}({\bf 1})\geq D_n,
 \qquad
 D_n=\frac1{n+1}\binom n{s_n},
\]
where
\(
 s_n=\lceil Q_n\rceil.
\)
Moreover, after increasing $N$ if necessary,
\[
 \log D_n\geq \frac P4\sqrt n(\log n)^2.
\]
\end{lemma}

\begin{proof}
Put
\(
 a=n-\lambda_1,
 b=n-\lambda'_1.
\)
If $\lambda\in{\mathscr X}_n^{Q_n}$, then $a\geq Q_n$ and $b\geq Q_n$.
Since $a$ and $b$ are integers, this means
\(
 a\geq s_n,
 b\geq s_n.
\)
For all sufficiently large $n$, we have $s_n<n/2$.

First suppose
\(
 a<\frac{n-1}{2}.
\)
By Lemma~\ref{two-row-min-degree},
applied with first-row defect $a$, we have
\(
 \chi^{\lambda}({\bf 1})\geq \chi^{(n-a,a)}({\bf 1}).
\)
The hook-length formula gives
\[
 \chi^{(n-a,a)}({\bf 1})=\binom na\frac{n+1-2a}{n+1-a}.
\]
Since $a<(n-1)/2$, the numerator $n+1-2a$ is at least $2$, and the
denominator is at most $n+1$.  Hence
\[
 \frac{n+1-2a}{n+1-a}\geq\frac1{n+1}.
\]
Also $s_n\leq a<n/2$, and the binomial coefficients increase in this range.
Therefore
\[
 \chi^{\lambda}({\bf 1})\geq \frac1{n+1}\binom n{s_n}=D_n.
\]

If instead
\(
 b<\frac{n-1}{2},
\)
we apply the same argument to the conjugate diagram $\lambda'$.  Since
$\chi^{\lambda'}({\bf 1})=\chi^{\lambda}({\bf 1})$, we also have $\chi^{\lambda}({\bf 1})\geq D_n$.

It remains to consider the case
\(
 a\geq\frac{n-1}{2},
 b\geq\frac{n-1}{2}.
\)
Then
\(
 \lambda_1\leq\frac{n+1}{2},
 \lambda'_1\leq\frac{n+1}{2}.
\)
For all sufficiently large $n$, both are at most $4n/7$.  Applying
Lemma~\ref{uniform-exponential-lower-bound} with $k=7/4$ and $B=3/2$, we
obtain
\[
 \chi^{\lambda}({\bf 1})\geq \left(\frac32\right)^n.
\]
On the other hand,
\[
 D_n\leq \binom n{s_n}\leq \left(\frac{{\rm e}n}{s_n}\right)^{s_n}.
\]
Since $s_n=O_P(\sqrt n\log n)$, the logarithm of the right-hand side is
$O_P(\sqrt n(\log n)^2)=o(n)$.  Therefore, for all sufficiently large $n$,
\[
 \left(\frac32\right)^n\geq D_n.
\]
This proves the lower bound $\chi^{\lambda}({\bf 1})\geq D_n$ in all cases.

Finally, we prove the explicit lower bound for $\log D_n$.  The elementary
estimate
\(
 \binom n{s_n}\geq \left(\frac n{s_n}\right)^{s_n}
\)
gives
\[
 \log D_n\geq s_n\log\frac n{s_n}-\log(n+1).
\]
For all sufficiently large $n$, we have
\(
 Q_n\leq s_n\leq 2Q_n,
\)
and hence
\[
 \log\frac n{s_n}
 \geq
 \log\frac n{2Q_n}
 =
 \frac12\log n-\log(2P\log n)
 \geq
 \frac13\log n.
\]
Therefore
\[
 \log D_n
 \geq
 \frac P3\sqrt n(\log n)^2-\log(n+1).
\]
After increasing $N$, the last logarithmic term is absorbed, and we obtain
\[
 \log D_n\geq \frac P4\sqrt n(\log n)^2.
\]
\end{proof}

\begin{lemma}\label{SG-LS-ratio}
Fix a real number $H>0$. Let $\{m_n\}_{n\in {\mathbb N}}$ be a series of positive integers such that $\lim_{n\rightarrow \infty} \frac{m_n}{\sqrt{n}}=H$. There exists $N=N(H,P)$ such that, for all $n>N$, for every conjugacy class
$C$ with  $m_n$ fixed points and with no cycles of integer length
$2\leq i<Q_n$, and for every partition $\lambda\in {\mathscr Y}_n$,
\[
 \left|\frac{\chi^\lambda(C)}{\chi^{\lambda}({\bf 1})}\right|
 \leq
 (\chi^{\lambda}({\bf 1}))^{-3/8}.
\]
\end{lemma}

\begin{proof}
Let $\sigma\in C$.  We estimate the Larsen--Shalev orbit-growth invariant
$E(\sigma)$.  Let $\Sigma_k$ be the union of all $\sigma$-orbits of length at
most $k$, and let $e_1,e_2,\ldots$ be defined by
\[
 n^{e_1+\cdots+e_k}=\max(|\Sigma_k|,1).
\]
Since $\sigma$ has exactly $m_n$ fixed points and no cycles of integer
length $2\leq i<Q_n$, we have
\[
 |\Sigma_k|=m
 \qquad(1\leq k<Q_n).
\]
Thus $
 e_1=\frac{\log m_n}{\log n}
,$
as $m_n/\sqrt{n}$ tends to $H$, for any $\varepsilon>0$, for $n$ large enough, we can assume that $\log m_n \leq \frac{1}{2}\log n + \log H + \varepsilon,$ hence 
$$e_1 \leq \frac 12 + \frac{\log H+\varepsilon}{\log n}.$$
On the other hand, all orbit-growth mass beyond $e_1$ occurs at integer indices $k\geq Q_n$.
Since the total mass is $1$,
\[
 E(\sigma)=\sum_{k\geq1}\frac{e_k}{k}
 \leq
 e_1+\frac1{Q_n}(1-e_1).
\]
Therefore, for all sufficiently large $n$,
\[
 E(\sigma)\leq \frac12+\frac1{16}.
\]
By Larsen--Shalev~\cite[Theorem~1.1(i)]{LarsenShalev2008}, applied
with error $1/16$, after increasing $N$ if necessary we have, for all
$\lambda \in {\mathscr Y}_n$,
\[
 |\chi^\lambda(C)|
 \leq
 (\chi^{\lambda}({\bf 1}))^{E(\sigma)+1/16}
 \leq
 (\chi^{\lambda}({\bf 1}))^{5/8}.
\]
Dividing by $\chi^{\lambda}({\bf 1})$ gives the desired ratio estimate.
\end{proof}

\begin{proof}[Proof of Proposition \ref{SG-X-tail-estimate}]
By Lemma~\ref{SG-LS-ratio}, applied to $C_2$, for all sufficiently large $n$,
\[
 \left|\frac{\chi^\lambda(C_2)}{\chi^{\lambda}({\bf 1})}\right|
 \leq (\chi^{\lambda}({\bf 1}))^{-3/8}
 \qquad(\lambda\vdash n).
\]
Therefore
\begin{align*}
 |{\bf X}_n^{Q_n}|
 &=
 \left|
 \sum_{\lambda\in{\mathscr X}_n^{Q_n}}
 \frac{\chi^\lambda(C_1)\chi^\lambda(C_2)\chi^\lambda(C_3)}{\chi^{\lambda}({\bf 1})}
 \right| \leq
 \sum_{\lambda\in{\mathscr X}_n^{Q_n}}
 |\chi^\lambda(C_1)|\,|\chi^\lambda(C_3)|
 \left|\frac{\chi^\lambda(C_2)}{\chi^{\lambda}({\bf 1})}\right| \\
 &\leq
 D_n^{-3/8}
 \sum_{\lambda\in{\mathscr X}_n^{Q_n}}
 |\chi^\lambda(C_1)|\,|\chi^\lambda(C_3)|,
\end{align*}
where Lemma~\ref{SG-degree-lower} was used in the last step.

By the Cauchy--Schwarz inequality and the column orthogonality relations,
\begin{align*}
&\sum_{\lambda\in\mathscr{X}_n^{Q_n}} |\chi^\lambda(C_1)\chi^\lambda(C_3)|
\le \Bigl(\sum_{\lambda\in\mathscr{X}_n^{Q_n}} |\chi^\lambda(C_1)|^2\Bigr)^{\!1/2}
      \Bigl(\sum_{\lambda\in\mathscr{X}_n^{Q_n}} |\chi^\lambda(C_3)|^2\Bigr)^{\!1/2} \\
&\le \Bigl(\sum_{\lambda\in\mathscr{Y}_n} |\chi^\lambda(C_1)|^2\Bigr)^{\!1/2}
      \Bigl(\sum_{\lambda\in\mathscr{Y}_n} |\chi^\lambda(C_3)|^2\Bigr)^{\!1/2} 
= \bigl(z(C_1)\,z(C_3)\bigr)^{1/2},
\end{align*}
where the second inequality follows by enlarging the index set to all of $\mathscr{Y}_n$.
Hence
\[
 |{\bf X}_n^{Q_n}|
 \leq
 D_n^{-3/8}\bigl(z(C_1)z(C_3)\bigr)^{1/2}.
\]
By Lemma~\ref{SG-centralizer},
\[
 \frac12\bigl(\log z(C_1)+\log z(C_3)\bigr)
 \leq
  H_1\sqrt n\log n+O_{H_1,P}(\sqrt n),
\]
while Lemma~\ref{SG-degree-lower} gives
\[
 \log D_n\geq \frac P4\sqrt n(\log n)^2.
\]
Hence
\begin{align*}
 \log |{\bf X}_n^{Q_n}|
 &\leq
 -\frac38\log D_n
 +\frac12\bigl(\log z(C_1)+\log z(C_3)\bigr) \\
 &\leq
 -\frac{3P}{32}\sqrt n(\log n)^2
 + H_1\sqrt n\log n+O_{H_1,P}(\sqrt n).
\end{align*}

The right hand side tends to $-\infty$ and Proposition \ref{SG-X-tail-estimate} follows.

\end{proof}

\begin{proof}[Proof of Theorem~\ref{SimiGaussian}]
For all sufficiently large $n$, we have $Q_n<(n-1)/2$.  Hence Lemma~\ref{ksep}
gives
\[
 {\bf Y}_n
 =2\sum_{0\leq t<Q_n}{\bf Z}_n^t+{\bf X}_n^{Q_n}.
\]
Let $\varepsilon>0$.  By Lemma~\ref{SG-main-layers-estimate}, there exists
$N_1=N_1(\varepsilon,H_1,H_2,P)$ such that, for all $n>N_1$ and all triples
satisfying the conditions 1-4 in Theorem \ref{SimiGaussian}, for $(H_1,P)$, 
\[
 \left|\sum_{0\leq t<Q_n}{\bf Z}_n^t-e^{-H_1H_2}\right|<\frac{\varepsilon}{3}.
\]
By Lemma~\ref{SG-X-tail-estimate}, there exists
$N_2=N_2(\varepsilon,H_1,H_2,P)$ such that, for all $n>N_2$ and all such triples,
\[
 |{\bf X}_n^{Q_n}|<\frac{\varepsilon}{3}.
\]
Choose $N$ larger than $N_1,N_2$, and large enough so that
$Q_n<(n-1)/2$.  Then
$$
 \left|{\bf Y}_n(C_1,C_2,C_3)-2e^{-H_1H_2}\right|
 \leq
 2\left|\sum_{0\leq t<Q_n}{\bf Z}_n^t-e^{-H_1H_2}\right|
 +|{\bf X}_n^{Q_n}| \\
 <\varepsilon.
$$
This proves the theorem.
\end{proof}

\section{The Limit for Conjugacy Classes with Few Short Cycles}\label{section3}

We give a proof of Proposition \ref{limit2}.

\begin{proof}[Proof of Proposition \ref{limit2}]
Choose $\delta>0$ so small that
\[
 -\frac14+3\delta\leq -\frac15.
\]
For instance, one may take $\delta=1/100$.  Since each $C_{i,n}$  has $\theta_1(C_{i,n})=o(1)$,  by Larsen--Shalev~\cite[Theorem~1.2(ii)]{LarsenShalev2008},
applied with $m=4$, there is an integer $N$ such that, for all $n>N$,
for all such conjugacy classes and all partitions $\lambda\in {\mathscr Y}_n$,
\[
 |\chi^\lambda(C_i)|\leq (\chi^{\lambda}({\bf 1}))^{1/4+\delta}
 \qquad (i=1,2,3).
\]
Multiplying the three estimates and dividing by $\chi^{\lambda}({\bf 1})$ gives
\[
 \left|
 \frac{\chi^\lambda(C_1)\chi^\lambda(C_2)\chi^\lambda(C_3)}
 {\chi^{\lambda}({\bf 1})}
 \right|
 \leq
 (\chi^{\lambda}({\bf 1}))^{-1/4+3\delta}
 \leq
 (\chi^{\lambda}({\bf 1}))^{-1/5}.
\]
Hence 
\begin{align*}&|{\bf Y}(C_{1,n},C_{2,n},C_{3,n})-2|= \left| \sum_{\lambda \in {\mathscr Y}_n \backslash \{(n),(1^n)\}} \frac{\chi^\lambda(C_1)\chi^\lambda(C_2)\chi^\lambda(C_3)}
 {\chi^{\lambda}({\bf 1})} \right|\\
 &\leq  \sum_{\lambda \in {\mathscr Y}_n \backslash \{(n),(1^n)\}}\left| \frac{\chi^\lambda(C_1)\chi^\lambda(C_2)\chi^\lambda(C_3)}
 {\chi^{\lambda}({\bf 1})} \right| \leq \sum_{\lambda \in {\mathscr Y}_n \backslash \{(n),(1^n)\}}(\chi^{\lambda}({\bf 1}))^{-1/5}.
\end{align*}
By \cite[Theorem 1.1]{Liebeck2004}, the Witten zeta function 
$$\sum_{\lambda \in {\mathscr Y}_n \backslash \{(n),(1^n)\}} (\chi^{\lambda}({\bf 1}))^{-\frac{1}{5}}=\sum_{\lambda \in {\mathscr Y}_n} (\chi^{\lambda}({\bf 1}))^{-\frac{1}{5}}-2=O(n^{-\frac{1}{5}}),$$
which proves the proposition.
\end{proof}

\section{Conclusions and Further Discussions}

In the final part of this paper we present a unified discussion of the two main results, from which the critical significance of the scale $\sqrt n$ (i.e., $n^\alpha$ with $\alpha=1/2$) for the number of fixed points becomes apparent.  
Recall the decomposition  
\[
\mathbf{Y}_n = 2\sum_{t<Q_n} \mathbf{Z}_n^t + \mathbf{X}_n^{Q_n}.
\]
We already observed that if the only cycles shorter than $Q_n$ in $C_1$ and $C_2$ are the fixed points, then each $\mathbf{Z}_n^t$ receives a contribution only from the hook diagram $(n-t,1^t)$, namely  
\[
q_n(t) = \frac{(-1)^t}{t!}\,\mu_n A_n(t),
\]
where $\mu_n = (m_{1,n}-1)(m_{2,n}-1)/(n-1)$.  As a rough approximation we replace $\mu_n A_n(t)$ by the leading term $\frac{m_{1,n}m_{2,n}}{n} \sim H_1H_2\, n^{2\alpha-1}$.

If we assume for a moment that $\mathbf{X}_n^{Q_n}$ remains a negligible tail, then the main term behaves like  
\[
\mathbf{Y}_n \sim 2\exp\bigl(-H_1H_2\, n^{2\alpha-1}\bigr).
\]
Hence, when $\alpha < 1/2$ we have $n^{2\alpha-1}\to 0$, which forces $\mathbf{Y}_n\to 2$; this is exactly Proposition~\ref{limit2}.  
The critical case is $\alpha = 1/2$, where $\mathbf{Y}_n \to 2e^{-H_1H_2}$; this is Theorem~\ref{SimiGaussian}.  
For $\alpha > 1/2$ we obtain $\sum_{t<Q_n} q_n(t) \to 0$.

In this sense $\sqrt n$ is a truly critical scale.  Moreover, the character sum undergoes a genuine phase transition at this point.

We deliberately state only that $\sum_{t<Q_n} q_n(t)\to 0$ when $\alpha>1/2$, without drawing a conclusion about $\mathbf{Y}_n$ itself.  The reason is that as $\alpha$ keeps increasing, $\mathbf{X}_n^{Q_n}$ will eventually cease to be a tail term.  This can already be seen from the expression of $q_n(t)$: the principal contributions come from hook diagrams, and when the number of fixed points is relatively small, almost no hooks can fall inside the square $Q_n \sim \sqrt n\log n$.  However, once $\alpha>1/2$ grows further, a large number of hooks will enter $\mathbf{X}_n^{Q_n}$. Indeed, a hook diagram falls into $\mathscr{X}_n^{Q_n}$ precisely when its row
defect satisfies $t = n-\lambda_1 \ge Q_n$, while the column defect
$n-\lambda_1' = n-t-1$ automatically stays above $Q_n$ as long as
$t\le n/2$, which is always the case in the range under discussion.
When the number of fixed points is on the critical scale
$\alpha=1/2$, the peak is at $t=O(1)$, hence $t\ll Q_n$ and the hooks
remain in the main layers.  As $\alpha$ increases beyond $1/2$, the
peak shifts to larger~$t$, eventually crossing the threshold $Q_n$
when $\alpha>3/4$.  At that stage the trade-off between the main layers and the tail can no longer be detected by the Larsen–Shalev estimates used in this paper; with the present techniques the asymptotic behaviour of $\mathbf{X}_n^{Q_n}$ in this range remains open.

The asymptotic behaviour of the first Hurwitz condition is now fairly clear on both sides of the critical scale.  The second condition, transitivity of the monodromy group, behaves in a strikingly different way.  Exactly where Theorem~\ref{SimiGaussian} predicts a rich, parameter‑dependent limit, the transitivity problem becomes highly unstable: for the same asymptotic parameters one can find both realisable and non‑realisable triples, and a small perturbation of the fixed‑point sets can toggle between the two.  This contrast has motivated a long sequence of works on unrealisable branch data~\cite{Edmonds1984,Pakovich2009,Pervova2006,Pervova2008,Pascali2009,Pascali2012,Zheng2006}.  A recent complete enumeration up to degree~$32$ by Xu et al.\ \cite{Wang2025} reveals that \emph{all} known unrealisable triples contain a cycle of length $1,2$ or $3$.  

Proposition~\ref{limit2} treats the case where each conjugacy class contains at most $n^{o(1)}$ short cycles.  If we strengthen this hypothesis by requiring that $C_1, C_2, C_3$ are all $3$-derangements (i.e., contain no cycles of length $1,2$ or $3$), then the branching data are conjectured to be always realisable \cite{Zheng2006}.  Indeed, the conjecture formulated by Zheng covers this situation as a special case, and all known computational evidence supports it.  In other words, once short cycles are completely removed, the chaotic behaviour of the second condition is expected to disappear entirely, leaving a regular landscape where the first condition alone guarantees the existence of a transitive realisation.

For the critical scale $\sqrt n$, however, the situation is far more mysterious.  Unlike the $3$-derangement case, for the situation of Theorem~\ref{SimiGaussian} we possess neither a plausible conjecture nor a systematic computational exploration on this asymptotic scale.  The presence of a macroscopic number of fixed points, combined with the requirement that all other cycles be very long, creates a delicate combinatorial environment in which both transitive and intransitive solutions might coexist, or one of the two possibilities could be generic.  At present we have no indication which alternative prevails.  Resolving this puzzle would require either new constructions of unrealisable data with many fixed points but no short cycles, or a proof that such data cannot exist.

\subsection*{Acknowledgements} \

 B.X is supported in part by the Project of Stable Support for Youth Team in Basic Research Field, CAS (Grant No. YSBR-001) and NSFC (No.
12271495).

The authors gratefully acknowledge the use of several computational tools 
during the preparation of this manuscript. 
 GPT-5.5pro contributed to formulating the initial ideas for the uniform 
exponential lower bound (Lemma~\ref{uniform-exponential-lower-bound}). 
DeepSeek-R1 provided the framework for the discussion on the critical 
exponent $\alpha=1/2$ in the concluding remarks. 
Kimi-2.6 generated the TikZ diagrams for Figure~\ref{graph1} and 
Figure~\ref{balancing}, and provided substantial assistance with English 
language editing and the overall structural layout of the paper. 
All outputs from these tools were critically reviewed, verified, and 
revised by the authors, who retain full responsibility for the final 
content of this work.

{\sc 
School of Mathematics and Statistics, Hengyang Normal University, 
Hengyang 421001, China,}\quad {\tt liangdun@hynu.edu.cn}

{\sc School of Mathematical Sciences, University  of Science and Technology of China, Hefei 230026, China}, \quad {\tt bxu@ustc.edu.cn}

{\sc School of Mathematical Sciences, University  of Science and Technology of China, Hefei 230026, China,}\quad  {\tt ywy72@mail.ustc.edu.cn}

\end{document}